\documentclass[12pt,letterpaper]{amsart}
\usepackage{amsmath,amssymb,amsfonts}
\usepackage[all]{xy}
\usepackage{graphicx}
\usepackage{caption}
\usepackage{enumerate}
\usepackage{xcolor}
\usepackage{ulem}
\usepackage{ stmaryrd }

\usepackage{appendix}
\usepackage{a4wide}
\usepackage{hyperref}

\newtheorem{thm}{Theorem} 
\newtheorem{prop}{Proposition}
\newtheorem{lem}{Lemma}
\newtheorem{cor}{Corollary}

\theoremstyle{definition}

\newtheorem{rem}{Remark}

\DeclareMathOperator{\rank}{\mathrm{rank}}

\DeclareMathOperator{\alg}{\mathrm{alg}}

    \DeclareFontFamily{U}{wncy}{}
    \DeclareFontShape{U}{wncy}{m}{n}{<->wncyr10}{}
    \DeclareSymbolFont{mcy}{U}{wncy}{m}{n}
    \DeclareMathSymbol{\Sha}{\mathord}{mcy}{"58}

\numberwithin{equation}{section}

\DeclareSymbolFont{bbold}{U}{bbold}{m}{n}
\DeclareSymbolFontAlphabet{\mathbbold}{bbold}

\DeclareSymbolFont{bbold}{U}{bbold}{m}{n}
\DeclareSymbolFontAlphabet{\mathbbold}{bbold}
 

\title{Spectral perturbation by rank $m$ matrices}
\author{Jonathan L. Merzel, J\'an Min\'a\v{c}, Tung T. Nguyen, Federico W. Pasini}
\address{Department of Mathematics, Soka University of America, 1 University Drive, Aliso Viejo, CA 92656}
\email{jmerzel@soka.edu} 

\address{Department of Mathematics, Western University, London, Ontario, Canada N6A 5B7}
\email{minac@uwo.ca}

\address{Department of Mathematics, The University of Western Ontario, London, ON, Canada, N6A 5B7}
\email{tungnt@uchicago.edu}

\address{Huron University College}
\email{fpasini@uwo.ca }

\date{\today}
\begin{document}
\maketitle
\begin{abstract}
Let $A$ and $B$ designate $n\times n$ matrices with coefficients in a field $F$.  In this paper, we completely answer the following question:  For $A$ fixed, what are the possible characteristic polynomials of $A+B$, where $B$ ranges over matrices of rank $\le m$?

\end{abstract}
\section{Introduction} 
The perturbation of a given matrix by another low-rank matrix is an important topic in mathematics, physics, and engineering. For example, it has been used to study the stability and controllability of dynamical systems,  the Baik-Ben Arous-Péché (BBP) phase transition, and quantum chaotic scattering (see \cite{[BBP]}, \cite{[FS]}, \cite{[KN]}, \cite{[Peche]}, \cite{[Shinners]}). Consequently, the spectral perturbation problem has been extensively studied in the literature. For interested readers, we refer to some prominent works on this topic (see for example \cite{[BMRR]}, \cite{[Kru]}, \cite{[MMRR]}, \cite{[MR]}, \cite{[RW]}).

A particularly interesting question in the study of low-rank perturbation is the following: For a fixed $n \times n$ matrix $A$ with coefficients in a field $F$, what are the possible characteristic polynomials of $A+B$, where $B$ ranges over matrices with coefficients in $F$ and of rank $\le m$?  In this article, we answer this question completely without any restriction on the field $F$. 


To state our main theorem, we first introduce some notation. For a polynomial $p(x) \in F[x]$  and $\lambda \in \overline{F}$ (a fixed algebraic closure of $F$) we write $m_{\lambda}(p(x))$ for the multiplicity of $\lambda$ as a zero of $p(x)$ (taking this to be 0 if $\lambda$ is not a zero of $p(x)$). The characteristic polynomial of a matrix $A$ will be designated as $p_A$;
for $\lambda \in\overline{F}$ we denote by $\alg_{\lambda }(A)$ the algebraic
multiplicity of $\lambda $ as an eigenvalue of $A$, that is, \ $\alg_{\lambda
}(A)=$\ \ $m_{\lambda }(p_{A})$. Our principal result is the following. 
\begin{thm} 
Let $A$ be an $n\times n$ matrix over a field $F$ and $q(x) \in F[x]$ be monic of degree $n$.  Then there exists an $n\times n$ matrix $B$ over $F$ of rank $\le m$ such that $p_{A+B}=q$ if and only if for each eigenvalue $\lambda$ of A, \[  m_{\lambda}(q) \geq \alg_{\lambda}(A)-\sum_{j=1}^m k_{\lambda,j}  \]
where $k_{\lambda,1}\ge k_{\lambda,2} \ge \cdots \ge k_{\lambda,m}$ are the sizes of the largest $m$ Jordan blocks for $\lambda$ in the Jordan form for $A$.
\end{thm}
The structure of our article is as follows. In Section 2, we derive the necessary condition in the theorem for the existence of the matrix $B$, using a rank estimate. In Section 3, we show that the necessary condition is also sufficient. Additionally, we provide a concrete demonstration of our proof in the case $m=2$.

\section{Necessary conditions using the Jordan canonical form} 
As indicated above, we will fix a matrix $A$, and let $B$ be a matrix of rank less than or equal to the positive integer $m$.
In this section, we develop a necessary condition on a monic polynomial $q$ of degree $n$ for $q=p_{A+B}$ for some such $B$.

\begin{lem} 
The following identity holds: 
\[ (B+A)^k=\left[ \sum_{m=0}^{k-1} A^m B (B+A)^{k-m-1} \right]+A^k .\] 

\end{lem} 
\begin{rem}  The result in fact holds for arbitrary elements $A,B$ in any ring, as the proof below shows. We thank the referee for suggesting the following simplified version of our original proof.

\end{rem}
\begin{proof}
    We have 
    \begin{align*}
    (B+A)^k - A^k &= \sum_{i=0}^{k-1} \left[A^i (B+A)^{k-i}-A^{i+1}(B+A)^{k-i-1} \right] \\
                  &= \sum_{i=0}^{k-1} A^i((B+A)-A)(B+A)^{k-i-1} \\
                  &= \sum_{i=0}^{k-1} A^i B (B+A)^{k-i-1}.
    \end{align*}
\end{proof} 
We provide another pictorial proof for Lemma 1.

\begin{proof}
Terms in the expression of $(A+B)^k$ correspond to paths of length $k-1$ in the following labeled graph (with $2k$ nodes).

\xymatrix{
&&&& A \ar[r] \ar[rd] & A \ar[r] \ar[rd] & A \ldots  \ar[r] \ar[rd] & A\\
&&&& B \ar[r] \ar[ru] & B \ar[r] \ar[ru] & B \ldots \ar[r] \ar[ru] & B \\
}
Apart from the term $A^k$, each term has an initial block of the form $A^mB$, $0 \leq m \leq k-1$. Visualizing that block in the graph above (starting from the left), and considering all terms which begin with that block, we see that they correspond to continuing paths through the expansion $(B+A)^{k-m-1}$.
\end{proof}
\begin{cor} \label{cor:rank_estimate}
For each $k$ 
\[ \rank((A+B)^k) \leq k \rank(B)+\rank(A^k).\]

\end{cor} 

\begin{proof}
This is a direct consequence of Lemma 1 and the facts that for two matrices $M,N$ 
\[ \rank(MN) \leq \min \{\rank(M), \rank(N) \}, \]
and 
\[ \rank(M+N) \leq \rank(M)+\rank(N) .\] 
\end{proof}

Let $C$ be a matrix defined over $F$ and $\lambda \in \overline{F}$. As in the statement of the main theorem, we denote by $\alg_{\lambda}(C)$ the algebraic multiplicity of $\lambda$ with respect to $C$. More precisely, 
\[ \alg_{\lambda}(C)=m_{\lambda}(p_{C}(x)).\]
Let $k_{ \lambda,1} \geq k_{\lambda, 2}  \geq \ldots $ be the sizes of all Jordan blocks of $A$ with $\lambda$ on the diagonal (here, to avoid burdening the notation, we do not explicitly fix the number of Jordan blocks, but we think of $k_{\lambda,j}$ as an eventually zero sequence of integers or equivalently we adopt the convention that a $0 \times 0$ Jordan block is an empty block).

By Corollary \ref{cor:rank_estimate}, we have 
\[ \rank((A+B-\lambda I_n)^{k_{\lambda,i}}) \leq k_{\lambda,i} \rank(B)+\rank((A-\lambda I_n)^{k_{\lambda,i}}) \leq m k_{\lambda,i} +\rank((A-\lambda I_n)^{k_{\lambda,i}}) .\]
It is straightforward to see that 
\begin{align*}
 \rank((A-\lambda I_n)^{k_{\lambda,i}})&=(n-\alg_{\lambda}(A))+\sum_{j=1}^{i} (k_{\lambda,j}-k_{\lambda,i}). \\
\end{align*}

Therefore 
\begin{equation}\label{eq:rank estimate}
    \rank((A+B-\lambda I_n)^{k_{\lambda,i}}) \leq n-\alg_{\lambda}(A)+ (m-i)k_{\lambda,i} +\sum_{j=1}^i k_{\lambda,j}.
\end{equation}

\begin{rem} \label{rem:min}
Our goal here is to make sure that the right-hand side is as small as possible as a function of $i$. 
Equivalently, we want to minimize the sum $s_i=(m-i)k_{\lambda,i}+\sum_{j=1}^i k_{\lambda,j}$. We claim that this sum attains its minimum at $i=m$. In fact, we have 
\begin{align*}
s_i-s_{i+1} &= \left[(m-i)k_{\lambda,i}+\sum_{j=1}^i k_{\lambda,j} \right] - \left[(m-i-1)k_{\lambda,i+1}+\sum_{j=1}^{i+1} k_{\lambda,j} \right] \\
&=(m-i)(k_{\lambda,i}-k_{\lambda,i+1}),
\end{align*}
From this equality, we see that the sequence $s_i$ is nonincreasing for $i \leq m$ and nondecreasing for $i \geq m$. It, therefore, attains its minimum at $i=m$.
\end{rem}

Taking $i=m$ in estimate \eqref{eq:rank estimate}, which is the optimal choice by Remark \ref{rem:min}, we see that 
\[ \rank((A+B-\lambda I_n)^{k_{\lambda,m}}) \leq n-\alg_{\lambda}(A) +\sum_{j=1}^m k_{\lambda,j}.\]
Consequently, we obtain: 

\begin{prop} \label{prop:condition} 
Let $A$ be a given matrix. Suppose $B$ is a matrix with rank at most $m$ such that $p_{A+B}(x)=q(x)$. Then 
\begin{equation}\label{eq:unnecessary} 
 \text{for each eigenvalue $\lambda$ of A, }m_{\lambda}(q) \geq \alg_{\lambda}(A)-\sum_{j=1}^m k_{\lambda,j} . \end{equation}
\end{prop} 

\begin{proof}
We remark that if $k_{\lambda, m}=0$ then the above statement is trivially true. When $k_{\lambda, m} \neq 0$, we have 
\begin{align*}
m_\lambda(q) &=m_\lambda(p_{A+B})=\alg_\lambda(A+B)=\alg_0(A+B-\lambda I_n)=\alg_0((A+B-\lambda I_n)^{k_{\lambda,m}})\\
&\ge n-\rank((A+B-\lambda I_n)^{k_{\lambda,m}})\\
&\ge \alg_{\lambda}(A)-\sum_{j=1}^m k_{\lambda,j}.
\end{align*}
\end{proof}

By the courtesy of the referee, we provide below their alternative proof for Proposition \ref{prop:condition}. We thank the referee for sharing this proof.
\begin{proof}
    It suffices to consider the situation of the eigenvalue $0$. Set $s: = \sum_{i=m+1}^{\infty} k_{0,i}(A)$ and $j \in \llbracket 0, s-1 \rrbracket$. It suffices to prove that the coefficient of $p_{A+B}$ on $x^j$ equals $0.$ To do so, we see $A$ and $B$ as matrices with entries in $\overline{F}$ and reduce the situation to the case where $A$ is in Jordan canonical form, with the first diagonal blocks nilpotent and with respective sizes $k_{0,1}(A), k_{0,2}(A), \ldots$ We set $J_r :=  \llbracket 1 + \sum_{i=1}^{r-1} k_{0, i}(A), \sum_{i=1}^r k_{0,i}(A) \rrbracket$ for $r \geq 1.$

    Classically, the coefficient of $p_{A+B}$ on $x^j$ is the sum of all principal $(n-j) \times (n-j)$ minors of $A +B.$ We simply prove that all these minors are equal to zero. So, let $I \subseteq \llbracket 1, n \rrbracket$ be of cardinality $n-j$, and denote by $A_{I}, B_{I},$ and $(A+B)_{I}$ the corresponding principal submatrices. Now the key is to note that $\rank(A_{I})+\rank(B_I) < n-j.$ Simply, $\rank(B_I) \leq m$ on the one hand, and on the other hand, seeing that $A_I$ is block diagonal with its diagonal blocks being principle submatrices of the Jordan cells of $A$, we see that its null space has dimension greater than or equal to the number of integers $r \geq 1$ such that $I \cap J_{r} \neq \emptyset$ (note that every principal submatrix of a Jordan cell is singular because its first column equals zero). Now, assuming that $\rank(A_I) + \rank(B_I) \geq n-j$, we would obtain $\rank(A_I) \geq n-j-m$, and hence there would be at most $m$ integers $r \geq 1$ such that $ I \cap J_{r} \neq \emptyset$; obviously, because $(k_{0,i}(A))_{i}$ is in non-increasing order this would yield $|I| \leq n- \sum_{k=m+1}^{\infty} k_{0,i}(A)=n-s$, contradicting the assumption that $|I| = n-j>n-s.$
\end{proof}

\section{Sufficient conditions using the rational canonical form}
We now show that condition \eqref{eq:unnecessary} is sufficient for the existence of $B$ defined over $F$ (of rank $\le m$) with $p_{A+B}=q$.
As above, $A$ is a fixed $n\times n$ matrix over $F$  and $q(x)\in F[x]$ is monic of degree $n
$.

\bigskip

Assume now without loss of generality that $A$ is in rational canonical form, 
\begin{equation*}
A=\left[ 
\begin{array}{llll}
\boxed{p_s} &  &  & 0 \\ 
& \boxed{p_{s-1}} &  &  \\ 
&  & \ddots &  \\ 
0 &  &  & \boxed{p_1}%
\end{array}%
\right]
\end{equation*}

where $p_{1} \mid p_{2} \mid \cdots \mid p_{s}$ and \boxed{p_i} is the companion matrix of 
$p_{i}$.  (Note that $p_{1},\dots ,p_{s}$ $\in F[x]$.) We have $%
p_A=\prod\limits_{i=1}^{s}p_{i}$.

We first reformulate the necessary condition \eqref{eq:unnecessary} in terms of $p_{1},\dots ,p_{s}$. 
\begin{prop} \label{prop:condition_rational} 
For $q(x) \in F[x]$, the condition  \eqref{eq:unnecessary}
is equivalent to $p_1 p_2 \cdots p_{s-m} \mid q.$

\end{prop}
\begin{proof}

The Jordan form for $A$ is the direct sum of
Jordan blocks from the Jordan decompositions of the \boxed{p_i} . \ But %
\boxed{p_i} has $p_{i}$ as its minimal and characteristic polynomial, and so
there can be at most one Jordan block with a given eigenvalue in the Jordan
decomposition of \boxed{p_i} . \ Since $p_{1} \mid p_{2}\mid \cdots \mid p_{s}$, the
largest $m$ Jordan  blocks for an eigenvalue $\lambda $ come from $\boxed{p_s},\dots, \boxed{p_{s-m+1}} $; thus
\[ \sum_{j=1}^m k_{\lambda,j}=\sum_{j=s-m+1}^s m_\lambda(p_j)\]
while
\[ \alg_{\lambda}(A)=\sum_{j=1}^s m_\lambda(p_j).\]

 So \eqref{eq:unnecessary} is equivalent to $m_{\lambda }(q)\geq
\sum\limits_{i=1}^{s-m}m_{\lambda }(p_{i})$ for each eigenvalue $\lambda$ of $A$ . \ Since the $p_{i}$'s \ have
only eigenvalues of $A$ as roots, this amounts to $p_{1}\cdots p_{s-m}\mid q$
\end{proof}
\begin{prop} \label{prop:sufficient}
If $q(x) \in F[x]$ is monic of degree $n$ and satisfies condition \eqref{eq:unnecessary} in Proposition \ref{prop:condition} then there exists a matrix $B$ over $F$ of rank at most $m$ with $p_{A+B}=q$.
\end{prop}
\begin{proof}

 If condition \eqref{eq:unnecessary} holds then by Proposition \ref{prop:condition_rational} we have $p_{1}\cdots p_{s-m}\mid q$; set $h=q/(p_{1}\cdots p_{s-m})$. 
Let $d_i$ be the degree of $p_i$ for $i=1,\dots,s$.

Certainly our goal is accomplished if we are able to create a matrix with
characteristic polynomial  $q$ by replacing $m$ columns of $A$ with new
columns whose entries are in $F$.  Let
$A_{i\text{ }}$ be the $i$th column of $A$ and let $e_{i}$ be the column
vector with $1$ in position $i$ and $0$ elsewhere. Also for $1\leq i\leq m$
let $\delta _{i}=\sum\limits_{j=0}^{i-1}d_{s-j}$. \ \ (Note that $\deg
h=\delta _{m}$.) \ We claim that we can alter columns $\delta _{1},\delta
_{2},\dots ,\delta _{m}$ \ of $A$\ so that the first $\delta _{m}$ rows and
columns constitute \boxed{h}, the companion matrix of $h$. \  For each $%
i\notin \{\delta _{1},\delta _{2},\dots ,\delta _{m}\}$ where $1\leq i\leq
\delta _{m}$ we already have $A_{i\text{ }}=e_{i+1}.$ \ To create \boxed{h}
we need only replace $A_{i\text{ }}$with $e_{i+1}$ for $i=\delta _{1},\delta
_{2},\dots ,\delta _{m-1}$ and replace $A_{\delta _{m}\text{ }}$with $%
[-b_{0},-b_{1},\dots ,-b_{\delta _{m}-1},0,\dots ,0]^{T}$ where $%
h(x)=x^{\delta _{m}}+b_{\delta _{m}-1}x^{\delta _{m}-1}+\cdots +b_{1}x+b_{0}$%
.

The resulting matrix, though no longer necessarily in rational canonical
form, is the direct sum of blocks \boxed{h},\boxed{p_{s-m}},\dots,\boxed{p_1}
and so has characteristic polynomial $hp_{s-m}\cdots p_{1}=q$ as desired.
\end{proof}

It may be helpful to look at the above proof in the special cases $m=1$ and $m=2$.

In the case $m=1$ we have $q=hp_1p_2 \cdots p_{s-1}$. Setting $d=d_s$ which is the degree of both $p_s$ and $h$, we may write
\begin{equation*}
p_{s}=x^{d}+a_{d-1}x^{d-1}+\cdots +a_{1}x+a_{0}
\end{equation*}
and%
\begin{equation*}
h=x^{d}+b_{d-1}x^{d-1}+\cdots +b_{1}x+b_{0},
\end{equation*}%
So
\begin{equation*}
\boxed{p_s}=\left[ 
\begin{array}{lllll}
0 & 0 & \cdots & 0 & -a_{0} \\ 
1 & 0 & \cdots & 0 & -a_{1} \\ 
0 & 1 & \cdots & 0 & -a_{2} \\ 
\vdots & \vdots &  & \vdots & \vdots \\ 
0 & 0 & \cdots & 1 & -a_{d-1}%
\end{array}%
\right]
\end{equation*}%
and%
\begin{equation*}
\boxed{h}=\left[ 
\begin{array}{lllll}
0 & 0 & \cdots & 0 & -b_{0} \\ 
1 & 0 & \cdots & 0 & -b_{1} \\ 
0 & 1 & \cdots & 0 & -b_{2} \\ 
\vdots & \vdots &  & \vdots & \vdots \\ 
0 & 0 & \cdots & 1 & -b_{d-1}%
\end{array}%
\right]
\end{equation*}
Taking $B$ to be a matrix with $0$ everywhere but the first $d$ entries of
column $d$, and having for those entries $\left[ a_{0}-b_{0}\ \ \
a_{1}-b_{1}\ \ \cdots \ \ a_{d-1}-b_{d-1}\right] ^{T}$, we find in forming $%
A+B$ that we have simply replaced \boxed{p_s} with \boxed{h}, and it is clear
that $B$ has coefficients in $F$, that $B$ has rank one and that the
characteristic polynomial of $A+B$ is $hp_1p_2 \cdots p_{s-1}=q$.

In the case $m=2$ we have $ q=hp_{1}\cdots p_{s-2}$, and we can modify just 2 columns of the (rational
canonical form of the) matrix $A$ to transform the leftmost two blocks%
\begin{equation*}
\left[ 
\begin{array}{ll}
\boxed{p_s} &  \\ 
& \boxed{p_{s-1}}%
\end{array}%
\right] =\left[ 
\begin{array}{ll}
\begin{array}{lllll}
0 & 0 & \cdots & 0 & -a_{0}^{(s)} \\ 
1 & 0 & \cdots & 0 & -a_{1}^{(s)} \\ 
0 & 1 & \cdots & 0 & -a_{2}^{(s)} \\ 
\vdots & \vdots &  & \vdots & \vdots \\ 
0 & 0 & \cdots & 1 & -a_{d_{s}-1}^{(s)}%
\end{array}
& 
\begin{array}{lllll}
&  &  &  &  \\ 
&  &  &  &  \\ 
&  & 0 &  &  \\ 
&  &  &  &  \\ 
&  &  &  & 
\end{array}
\\ 
\begin{array}{lllll}
&  &  &  &  \\ 
&  &  &  &  \\ 
&  & 0 &  &  \\ 
&  &  &  &  \\ 
&  &  &  & 
\end{array}
& 
\begin{array}{lllll}
0 & 0 & \cdots & 0 & -a_{0}^{(s-1)} \\ 
1 & 0 & \cdots & 0 & -a_{1}^{(s-1)} \\ 
0 & 1 & \cdots & 0 & -a_{2}^{(s-1)} \\ 
\vdots & \vdots &  & \vdots & \vdots \\ 
0 & 0 & \cdots & 1 & -a_{d_{s-1}-1}^{(s-1)}%
\end{array}%
\end{array}%
\right]
\end{equation*}

into 
\begin{equation*}
\boxed{h}=\left[ 
\begin{array}{llllllllll}
0 & 0 & \cdots & 0 & 0 &  &  &  &  & -b_{0} \\ 
1 & 0 & \cdots & 0 & 0 &  &  &  &  & -b_{1} \\ 
0 & 1 & \cdots & 0 & 0 &  &  & 0 &  & -b_{2} \\ 
\vdots & \vdots &  & \vdots & \vdots &  &  &  &  & \vdots \\ 
0 & 0 & \cdots & 1 & 0 & 0 & 0 & \cdots & 0 & -b_{d_{s}-1} \\ 
0 & 0 & \cdots & 0 & 1 & 0 & 0 & \cdots & 0 & -b_{d_{s}} \\ 
&  &  & 0 & 0 & 1 & 0 & \cdots & 0 & -b_{d_{s}+1} \\ 
&  & 0 &  &  &  & \ddots &  &  & \vdots \\ 
&  &  &  &  &  &  &  &  &  \\ 
0 & 0 & \cdots & 0 & 0 & 0 & 0 & \cdots & 1 & -b_{d_{s}+d_{s-1}-1}%
\end{array}%
\right]
\end{equation*}%
where $p_{i}=x^{d_{i}}+a_{d_{i}-1}^{(i)} x^{d_{i}-1}+\cdots
+a_{0}^{(i)}$ and $%
h=x^{d_{s}+d_{s-1}}+b_{d_{s}+d_{s-1}-1}x^{d_{s}+d_{s-1}-1}+\cdots
+b_{1}x+b_{0}$. \ (The two altered columns are column $d_{s}$ and column $%
d_{s}+d_{s-1}$.) \ This yields a matrix $B$ of rank 2 such that the
characteristic polynomial of $A+B$ is $q$.

We can now prove our principal result.
\begin{proof}[Proof of Main Theorem]
Combine the sufficient condition of Proposition \ref{prop:sufficient} with the necessary condition of Proposition \ref{prop:condition}.
\end{proof}

\section*{Acknowledgements}
We are very grateful to Professor Lyle Muller's encouragement, discussions, and support via BrainsCAN, and the NSF through a NeuroNex award (\#201576). J.M. gratefully acknowledges the Natural Sciences and Engineering Research Council of Canada (NSERC) grant R0370A01 and the Western University Faculty of Science Distinguished Professorship in 2020-2021. Further, J.M., T.T.N., and F.W.P. gratefully acknowledge the support from the Western Academy For Advanced Research at Western University. Finally, we thank an anonymous referee for their help in improving the quality and clarity of this paper.

\end{document}